\documentclass[11pt, reqno, psamsfonts]{amsart}
\pdfoutput=1

\usepackage{amssymb}
\usepackage{amsthm}
\usepackage{amsmath}
\usepackage{latexsym}
\usepackage[T1]{fontenc}
\usepackage[utf8]{inputenc}
\usepackage[russian, french, english]{babel}
\usepackage{graphicx}
\usepackage{wrapfig}
\usepackage[justification=centering, labelfont=bf]{caption}
\usepackage{subcaption}
\usepackage{mathtools}
\usepackage[hidelinks]{hyperref}
\usepackage{amsbsy}
\usepackage{enumitem}
\usepackage{mathrsfs}
\usepackage{thmtools}
\usepackage{marginnote}
\usepackage{mathabx}
\usepackage{lmodern}
\usepackage[spacing=true,kerning=true,babel=true,tracking=true]{microtype}
\usepackage[shortcuts]{extdash}
\usepackage[foot]{amsaddr} 
\usepackage{framed}
\usepackage[left=1in,right=1in,top=1in,bottom=1in,bindingoffset=0cm]{geometry}
\usepackage[backend=biber, style=alphabetic, sorting=nyt, maxnames=100]{biblatex}
\title{Building Large Free Subshifts Using the Local Lemma}
\date{}
\author{Anton~Bernshteyn}
\address{Department of Mathematics, University of Illinois at Urbana--Champaign, IL, USA}
\email{bernsht2@illinois.edu}
\thanks{This research is supported in part by the Waldemar J., Barbara G., and Juliette Alexandra Trjitzinsky Fellowship.}

\newtheoremstyle{bfnote}%
{}{}%
{\slshape}{}%
{\bfseries}{\bfseries.}%
{ }%
{\thmname{#1}\thmnumber{ #2}\thmnote{ \ep{\normalfont{}#3}}}

\newtheoremstyle{defbfnote}%
{}{}%
{}{}%
{\bfseries}{.}%
{ }%
{\thmname{#1}\thmnumber{ #2}\thmnote{ (#3)}}

\newtheoremstyle{claim}%
{}{}%
{\slshape}{}%
{\itshape}{.}%
{ }%
{\thmname{#1}\thmnumber{ #2}\thmnote{ \ep{\normalfont{}#3}}}

\theoremstyle{bfnote}
\newtheorem{theo}{Theorem}[section]
\newtheorem{prop}[theo]{Proposition}
\newtheorem{lemma}[theo]{Lemma}
\newtheorem{corl}[theo]{Corollary}

\newtheorem{claim}[theo]{Claim}

\newcommand*{\myproofname}{Proof}

\theoremstyle{definition}

\newtheorem{ques}[theo]{Question}
\newtheorem*{remk*}{Remark}

\theoremstyle{remark}
\newtheorem*{ques*}{Question}

\newcommand{\0}{\varnothing}
\newcommand{\set}[1]{\{#1\}}

\newcommand{\dom}{\mathrm{dom}}

\newcommand{\acts}{\curvearrowright}
\newcommand{\pacts}{\,\tilde{\acts}\,}
\newcommand{\N}{\mathbb{N}}

\renewcommand{\U}{\mathscr{U}}

\renewcommand{\epsilon}{\varepsilon}
\renewcommand{\phi}{\varphi}
\renewcommand{\theta}{\vartheta}
\renewcommand{\leq}{\leqslant}
\renewcommand{\geq}{\geqslant}

\newcommand{\symdif}{\bigtriangleup}

\newcommand{\fins}[1]{[#1]^{<\infty}}
\newcommand{\finf}[2]{[#1 \to #2]^{<\infty}}
\renewcommand{\G}{\Gamma}

\newcommand{\B}{\mathscr{B}}

\newcommand{\defeq}{\coloneqq}
\newcommand{\rest}[2]{{{#1}\vert_{#2}}}
\newcommand{\emphd}[1]{{\fontseries{b}\selectfont\textsf{#1}}}
\newcommand{\St}{{\mathrm{St}}}
\renewcommand{\d}{\mathfrak{d}}
\renewcommand{\b}{\mathfrak{b}}
\newcommand{\w}{\mathfrak{w}}
\newcommand{\Prop}{\mathrm{Prop}}
\newcommand{\Ball}{\mathrm{B}}
\newcommand{\Col}[4]{{\mathrm{Col}_{#1,#2}(#4, #3)}}
\newcommand{\Forb}{\mathrm{Forb}}
\newcommand{\Nbhd}{\mathrm{N}}
\newcommand{\bemph}[1]{{\normalfont#1}} 
\newcommand{\ep}[1]{\bemph{(}#1\bemph{)}} 

\numberwithin{equation}{section}

\usepackage{etoolbox}

\patchcmd{\subsection}{\normalfont}{\itshape\bfseries}{}{}

\makeatletter
\def\@seccntformat#1{%
	\protect\textup{%
		\protect\@secnumfont
		\expandafter\protect\csname format#1\endcsname 
		\csname the#1\endcsname
		\protect\@secnumpunct
	}%
}

\makeatother

\makeatletter
\newcommand{\neutralize}[1]{\expandafter\let\csname c@#1\endcsname\count@}
\makeatother

\renewbibmacro{in:}{}

\renewbibmacro*{volume+number+eid}{%
	\printfield{volume}%
	\setunit*{\addnbspace}
	\printfield{number}%
	\setunit{\addcomma\space}%
	\printfield{eid}}

\DeclareFieldFormat[article]{volume}{\textbf{#1}\space}
\DeclareFieldFormat[article]{number}{\mkbibparens{#1}}

\DeclareFieldFormat{journaltitle}{#1,}
\DeclareFieldFormat[thesis]{title}{\mkbibemph{#1}\addperiod}
\DeclareFieldFormat[article, unpublished, thesis]{title}{\mkbibemph{#1},}
\DeclareFieldFormat[book]{title}{\mkbibemph{#1}\addperiod}
\DeclareFieldFormat[unpublished]{howpublished}{#1, }

\DeclareFieldFormat{pages}{#1}

\DeclareFieldFormat[article]{series}{Ser.~#1\addcomma}

\begin{document}
	
	\maketitle
	
	\begin{abstract}
		Gao, Jackson, and Seward~\cite{GJS1} proved that every countably infinite group $\G$ admits a nonempty free subshift $X \subseteq 2^\G$. Here we strengthen this result by showing that free subshifts can be ``large'' in various senses. Specifically, we prove that for any $k \geq 2$ and $h < \log_2 k$, there exists a free subshift $X \subseteq k^\G$ of Hausdorff dimension and, if $\G$ is sofic, entropy at least $h$, answering two questions attributed by Gao, Jackson, and Seward to Juan Souto \cite[Problems~11.2.4 and 11.2.5]{GJS2}. Furthermore, we establish a general lower bound on the largest ``size'' of a free subshift $X'$ contained inside a given subshift $X$. A central role in our arguments is played by the Lov\'asz Local Lemma, an important tool in probabilistic combinatorics, whose relevance to the problem of finding free subshifts was first recognized by Aubrun, Barbieri, and Thomass\'e~\cite{ABT}.
	\end{abstract}
	
	\section{Introduction}
	
	
	\noindent We use $\N$ to denote the set of all nonnegative integers and identify each $k \in \N$ with the set $\set{i \in \N \,:\, i < k}$. Countable sets are always assumed to carry discrete topologies. Throughout, $\G$ is a countably infinite group with identity element $\mathbf{1}$. For a set~$A$, the \emphd{shift action} $\G \acts A^\G$ on the set of all maps $x \colon \G \to A$ is defined via
	\[
	(\gamma \cdot x)(\delta) \defeq x(\delta \gamma) \qquad \text{for all } \gamma,\,\delta \in \G \text{ and } x \in A^\G.
	\]
	Occasionally, we will also have to consider the \emphd{right shift action} $A^\G \curvearrowleft \G$, defined similarly by
	\[
	(x \cdot \gamma)(\delta) \defeq x(\gamma\delta) \qquad \text{for all } \gamma,\,\delta \in \G \text{ and } x \in A^\G.
	\]
	Whenever we refer to the shift action, for instance when talking about shift-invariant sets, the \emph{left} shift action must be understood, unless explicitly stated otherwise. The main reason to invoke the right shift action is that the map $x \mapsto \St_\G(x)$ that associates to each point $x \in A^\G$ its stabilizer under the \emph{left} shift action is \emph{right}-shift-invariant; in other words, for all $\gamma \in \G$ and $x \in A^\G$, 
	\[
		\gamma \cdot x = x \,\Longrightarrow\, \gamma \cdot (x \cdot \delta) = x \cdot \delta \text{ for every } \delta \in \G.
	\]
	  If $A$ is a topological space, then the shift action $\G \acts A^\G$ is continuous with respect to the product topology. When $A = k \in \N$, the product space $k^\G$ is totally disconnected, compact, and metrizable. 
	
	A shift-invariant closed subset $X \subseteq k^\G$ is called a \emphd{subshift}. A subshift $X$ is \emphd{minimal} if $X \neq \0$ and there is no subshift $Y$ such that $\0 \neq Y \varsubsetneq X$.
	A subshift $X$ is \emphd{free} if the induced action $\G \acts X$ is free, i.e., if the stabilizer of every point $x \in X$ is trivial:
	\[
		\gamma \cdot x = x \,\Longrightarrow\, \gamma = \mathbf{1} \qquad \text{for all } \gamma \in \G \text{ and } x \in X. 
	\]
	Glasner and Uspenskij \cite[Problem~6.2]{GU} asked if every countable group admits a non\-emp\-ty free subshift and gave a positive answer for groups that are either Abelian or residually finite~\cite[Theorem~5.1]{GU}. Somewhat earlier, Dranishnikov and Schroeder~\cite[Theorem~2]{DS} reached the same conclusion for torsion\-/free hyperbolic groups. The problem was finally resolved in a tour de force by Gao, Jackson, and Seward~\cite{GJS1, GJS2}, who showed that not only do nonempty free subshifts exist for all groups, but they are rather numerous: For any $k \geq 2$, every nonempty shift\-/invariant open subset $U\subseteq k^\G$ contains continuumly many pairwise disjoint non\-emp\-ty free subshifts~\cite[Theorem~1.4.1]{GJS2}.
	
	
	Seward and Tucker-Drob~\cite{T-DS} further developed the techniques of \cite{GJS1, GJS2} in order to establish the following very strong result: If $\G \acts X$ is a free Borel action of $\G$ on a standard Borel space $X$, then there exists an equivariant Borel map $\pi \colon X \to 2^\Gamma$ such that $\overline{\pi(X)}$ is a free subshift~\cite[Theorem~1.1]{T-DS}. (Here, and in what follows, a horizontal line indicates topological closure.) This in particular implies that every countable group admits a free subshift \emph{with an invariant probability measure}, which answers a question raised by Gao, Jackson, and Seward~\cite[Problem~11.2.6]{GJS2}. Indeed, if the action $\G \acts X$ preserves a probability measure $\mu$, then the pushforward $\pi_\ast(\mu)$ is an invariant probability measure on $\overline{\pi(X)}$.
	

	
	For the rest of this article, fix an integer $k \geq 2$. We study how ``large,'' in various senses, a free subshift $X \subseteq k^\G$ can be. Specifically, we investigate the following questions, which are attributed by Gao, Jackson, and Seward to Juan Souto: 
	
	\begin{ques}[{\cite[Problem~11.2.5]{GJS2}}]\label{ques:Haus}
		For a given group $\G$, what is the largest possible Hausdorff dimension of a free subshift $X \subseteq k^\G$?
	\end{ques}
	
	\begin{ques}[{\cite[Problem~11.2.4]{GJS2}}]\label{ques:ent}
		For groups $\G$ in which a notion of entropy exists, what is the largest possible entropy of a free subshift $X \subseteq k^\G$?
	\end{ques}
	
	The notions of Hausdorff dimension and \ep{topological} entropy are reviewed in Section~\ref{sec:prelim}. To date, the largest class of groups for which a well-developed theory of entropy exists is formed by the so-called \emph{sofic} groups. Entropy for measure-preserving actions of sofic groups was introduced by Bowen~\cite{Bow0} and then extended to the topological setting by Kerr and Li~\cite{KerrLi}. For the smaller class of \emph{amenable} groups, entropy was introduced earlier by Keiffer~\cite{Keiffer} (with important further developments by Ornstein and Weiss~\cite{OrnsteinWeiss}) and is somewhat better behaved. Both the Hausdorff dimension and, if $\G$ is sofic, the entropy of $k^\G$ are equal to $\log_2 k$. 
	We answer Questions~\ref{ques:Haus} and \ref{ques:ent} by showing that 
	the Hausdorff dimension and, if $\G$ is sofic, the entropy of a free subshift can be made \emph{arbitrarily close} to this upper bound:
	
	\begin{theo}\label{theo:corl}
		{Let $U \subseteq k^\G$ be a nonempty shift-invariant open set. Then, for any $h < \log_2 k$:}
		\begin{enumerate}[label=\normalfont\ep{\roman*},align=left,labelindent=\parindent,leftmargin=*,itemsep=\parskip]
			\item {there exists a free minimal subshift $X \subseteq U$ of Hausdorff dimension at least~$h$;}
			
			\item {if $\G$ is amenable, then there exists a free minimal subshift $X \subseteq U$ of entropy at least~$h$;}
			
			\item {if $\G$ is sofic, then there exists a free subshift $X \subseteq U$ whose entropy with respect to any sofic approximation is at least $h$.}
		\end{enumerate}
	\end{theo}
	
	The main ingredient in our proof of Theorem~\ref{theo:corl} is the so-called \emph{Lov\'asz Local Lemma} (the~\emph{LLL} for short), a powerful tool in probabilistic combinatorics that is often used to prove existence results. A brief review of the~LLL is given in \S\ref{subsec:LLL}. Although the original proof of \cite[Theorem~1.4.1]{GJS2} due to Gao, Jackson, and Seward is quite technical, Aubrun, Barbieri, and Thomass\'e~\cite{ABT} later employed the~LLL to find a simple alternative construction of a nonempty free subshift $X \subseteq 2^\G$ for an arbitrary group $\G$. Elek~\cite{E}, following an approach based on nonrepetitive graph colorings and inspired by~\cite{AGHR}, obtained a new proof that there exist free subshifts with invariant probability measures under the assumptions that $\G$ is finitely generated and sofic; Elek's argument also relies heavily on the~LLL.
	
	The main result of this article is Theorem~\ref{theo:main}, of which Theorem~\ref{theo:corl} is a simple special case. We state Theorem~\ref{theo:main} in Section~\ref{sec:main_result} after introducing some necessary definitions. In the remainder of this introduction we give a brief informal overview of the statement of Theorem~\ref{theo:main} without being precise about the technical details.
	
	In this paper, we work with five notions of size for subshifts: Hausdorff dimension, entropy (with the cases of amenable and general sofic groups treated somewhat differently), width, pointwise width, and breadth. The former two are standard and reviewed in Section~\ref{sec:prelim}, while the latter three are defined in Section~\ref{sec:main_result} and are crucial for the statement of Theorem~\ref{theo:main}.
	
	The \emph{width} $\w(X)$ of a subset $X \subseteq k^\G$ is defined in a way that is quite similar to the definition of Hausdorff dimension, with the advantage of not requiring to choose a metric. It is not hard to see that for a subshift $X$, $\w(X)$ is a lower bound for the Hausdorff dimension and, if $\G$ is amenable, the entropy of $X$ (see Proposition~\ref{prop:comparison}\ref{item:Hausdorff},\ref{item:amenable}). The \emph{pointwise width} of a set $X$ is defined via
	\[
		\w^\ast(X) \defeq  \inf_{x \in X} \min \set{\w(\overline{\G \cdot x}), \, \w(\overline{x \cdot \G})}.
	\]
	If $X$ is a nonempty subshift with $\w^\ast(X) \geq h$, then we have $\w(X) \geq h$, and, moreover, $\w(Y) \geq h$ for {all} nonempty subshifts $Y \subseteq X$. Furthermore, if $\w^\ast(X)$ is sufficiently large, namely strictly higher than $(1/2)\log_2 k$, then $X$ must be \emph{free} (see Proposition~\ref{prop:comparison}\ref{item:free}). In view of these considerations, finding subshifts of large pointwise width becomes our primary objective.
	
	The last notion of size for a subshift $X$ that we introduce (and the last ingredient needed for the statement of our main result) is its \emph{breadth} $\b(X)$. The definition of breadth is directly informed by the requirements of the~LLL. In contrast to the other notions, it estimates the size of $X$ in a somewhat roundabout way: by measuring how ``small'' one can make a family of open sets whose translates cover the \emph{complement} of $X$. The main advantage of breadth is that it is usually easy to bound from below, since a lower bound on $\b(X)$ can be witnessed by a \emph{single} family $\U$ of open sets.
	
	Now we can state our main result:
	
	{
		\renewcommand{\thetheo}{\ref{theo:main}}
		\begin{theo}
			Let $X \subseteq k^\G$ be a subshift such that $\b(X) > 0$. Then $\w(X) \geq \b(X)$; moreover, for any $h < \b(X)$, there exists a nonempty subshift $X' \subseteq X$ with the following properties:
			\begin{enumerate}[label=\normalfont\ep{\roman*},align=left,labelindent=\parindent,leftmargin=*, itemsep=\parskip]
				\item the pointwise width of $X'$ is at least $h$;
				\item\label{item:main:sofic} if $\G$ is sofic, then the entropy of $X'$ with respect to any sofic approximation is at least~$h$;
				\item there exist an invariant probability measure $\mu$ on $X'$ and a factor map \[\pi \colon ([0;1]^\G, \lambda^\G) \to (X', \mu).\]
			\end{enumerate}
		\end{theo}
		\addtocounter{theo}{-1}
	}
	
	
	Theorem~\ref{theo:corl} easily follows from Theorem~\ref{theo:main} since $\b(k^\G) = \log_2 k$; the details are given in~\S\ref{subsec:proof_corl}.
	
	\subsection*{{Acknowledgements}}
	
	I am very grateful to Alexander Kechris, Robin Tucker-Drob, Benjamin Weiss, and the anonymous referee for their helpful comments and to Anush Tserunyan for insightful discussions.
	
	\section{Preliminaries}\label{sec:prelim}
	
	
	\subsection{{Basic definitions and notation}}\label{subsec:prelim}
	
	For convenience, we identify a function $f$ with its graph, i.e., with the set $\set{(x, y) \,:\, f(x) = y}$. This enables the use of standard set-theoretic notation, such as $\cup$, $|\cdot|$, $\subseteq$, etc., for functions. 
	The restriction of a function $f$ to a set $S$ is denoted by $\rest{f}{S}$. 
	Given sets $X$ and $Y$, we use $\fins{X}$ to denote the set of all finite subsets of $X$ and $\finf{X}{Y}$ to denote the set of all functions $\phi \colon F \to Y$ with $F \in \fins{X}$.
	The shift action $\G \acts k^\G$ naturally extends to an action $\G \acts \finf{\G}{k}$; specifically, for all $\phi \in \finf{\G}{k}$ and $\gamma \in \G$, let \[\dom(\gamma \cdot \phi) \defeq \dom(\phi) \gamma^{-1} \qquad \text{and} \qquad (\gamma \cdot \phi)(\delta) \defeq \phi(\delta \gamma) \text{ for all } \delta \in \dom(\gamma \cdot \phi)
	.\]
	The right shift action $\finf{\G}{k} \curvearrowleft \G$ is defined similarly in the obvious way.
	The topology on the space~$k^\G$ is generated by the \emphd{basic} open sets of the form
	\[
		U_\phi \defeq \set{x \in k^\G \,:\, x \supset \phi}, \qquad \text{where} \qquad \phi \in \finf{\G}{k} \setminus \set{\0}.
	\]
	Observe that each basic open set is also closed. 
	Note that the space $k^\G$ itself is \emph{not} a basic open set (this convention will simplify some of our definitions later). For $X \subseteq k^\G$ and $F \in \fins{\G}\setminus \set{\0}$, let
	\[
		X_F \defeq \set{\phi \in k^F \,:\, X \cap U_\phi \neq \0}= \set{\rest{x}{F} \,:\, x \in X}.
	\]
	We use $\lambda$ to denote the Lebesgue probability measure on the unit interval $[0;1]$.
	
	\subsection{{Hausdorff dimension}}\label{subsec:Hausdorff}
	
	To define Hausdorff dimension, we must first fix a metric on~$k^\G$. To that end, let $\gamma_0$, $\gamma_1$, \ldots{} be an arbitrary enumeration of the elements of $\G$. For distinct $x$, $y \in k^\G$, let
	\[
	\operatorname{dist}(x,y) \defeq 2^{-n}, \text{ where } n \defeq \min \set{i \in \N \,:\, x(\gamma_i) \neq y(\gamma_i) }.
	\]
	Of course, if $x = y$, then $\operatorname{dist}(x,y) \defeq 0$. Note that this metric is not shift-invariant and depends on the choice of the enumeration $\gamma_0$, $\gamma_1$, \ldots (in fact, the topology on $k^\G$ is not induced by any invariant metric).
	For $h \in [0;+\infty)$, the  \emphd{$h$-dimensional Hausdorff content} $C_h(X)$ of a set $X \subseteq k^\G$ is the infimum of all $\epsilon \in [0;+\infty)$ such that there is a cover $\mathscr{B}$ of $X$ by open balls with
	\[
	\sum_{B \in \mathscr{B}} \mathrm{diam}(B)^h = \epsilon.
	\]
	The \emphd{Hausdorff dimension} of $X$, denoted $\dim_H(X)$, is given by
	\[
		\dim_H(X) \defeq \inf \set{h \in [0;+\infty)\,:\, C_h(X) = 0}.
	\] 
	The following observation is easy:
	
	\begin{prop}
	We have $\dim_H(k^\G) = \log_2 k$ and $C_{\log_2k}(X) =0$ for any subshift $X \varsubsetneq k^\G$.
	\end{prop}
	
	\subsection{{Entropy for amenable groups}}\label{subsec:e_amenable}
	
	A group $\G$ is called \emphd{amenable} if it admits a \emphd{F\o lner sequence}, i.e., a sequence of nonempty finite subsets $(F_n)_{n=0}^\infty$ such that
	\[
	\lim_{n \to \infty} \frac{|(\gamma F_n) \symdif F_n|}{|F_n|} \,=\,0 \qquad \text{for all } \gamma \in \G.
	\]
	If $\G$ is amenable, then the \emphd{\ep{topological} entropy} $h(X)$ of a nonempty subshift $X \subseteq k^\G$ is given by
	\[
	h(X) \,\defeq\, \lim_{n \to \infty} \frac{\log_2 |X_{F_n}|}{|F_n|},
	\]
	where $(F_n)_{n=0}^\infty$ is a F\o lner sequence in $\G$. By a fundamental result of Ornstein and Weiss~\cite{OrnsteinWeiss}, the above limit always exists and is independent of the choice of the F\o lner sequence $(F_n)_{n=0}^\infty$. The entropy of a subshift obeys the following bounds:
	
	\begin{prop}\label{prop:bounds_amenable}
		If $\G$ is amenable, then $h(k^\G) = \log_2 k$ and $h(X) < \log_2 k$ for any subshift $X \varsubsetneq k^\G$.
	\end{prop}
	
	\subsection{{Entropy for sofic groups}}\label{subsec:e_sofic}
	
	A \emphd{pseudo-action} of $\G$ on a set $V$ is a map
	\[
		\alpha \colon \G \times V \to V \colon (\gamma, v) \mapsto \gamma \cdot_\alpha v.
	\]
	We write $\alpha \colon \G \pacts V$ to indicate that $\alpha$ is a pseudo-action of $\G$ on $V$. When $\alpha$ is understood, we usually simply write $\gamma \cdot v$ instead of $\gamma \cdot_\alpha v$.
	
	Let $\alpha \colon \G \pacts V$ be a pseudo-action of $\G$ on a set $V$ and let $F \subseteq \G$. An element $v \in V$ is \emphd{$F$-proper} \ep{with respect to~$\alpha$} if the following conditions are satisfied:
	\begin{itemize}[label=---]
		\item \emph{identity}: $\mathbf{1} \cdot v = v;$
		\item \emph{$F$-equivariance}: $\gamma \cdot (\delta \cdot v) = (\gamma\delta) \cdot v$ for all $\gamma$, $\delta \in F$ such that $\gamma\delta \in F$;
		\item \emph{$F$-freeness}: $\gamma \cdot v = \delta \cdot v \,\Longrightarrow\, \gamma = \delta$ for all $\gamma$, $\delta \in F$.
	\end{itemize}
	Let $\Prop_F(\alpha)$ denote the set of all $F$-proper elements $v \in V$. Note that $\alpha$ is a free action if and only if $\Prop_\G(\alpha) = V$. If $V$ is a finite set, then $\alpha$ is \emphd{$(\epsilon, F)$-faithful} for $\epsilon > 0$ and $F \subseteq \G$ if
	\[
		|\Prop_F(\alpha)| \geq (1-\epsilon)|V|. 
	\]
	
	A group $\G$ is called \emphd{sofic} if it admits a \emphd{sofic approximation}, i.e., a sequence $(\alpha_n)_{n = 0}^\infty$ of pseudo\-/actions on nonempty finite sets such that for all $\epsilon > 0$ and $F \in \fins{\G}$, all but finitely many of the pseudo-actions $\alpha_n$ are $(\epsilon, F)$-faithful.
	
	Sofic groups were introduced by Gromov~\cite{Gromov} as a common generalization of amenable and residually finite groups (the term ``sofic'' was coined somewhat later by Weiss~\cite{WeissSofic}). In a major breakthrough, Bowen~\cite{Bow0} generalized the notion of entropy from amenable to all sofic groups. Bowen's work was further extended by Kerr and Li~\cite{KerrLi}, who, in particular, introduced sofic entropy to the topological setting and proved the variational principle for actions of sofic groups. The presentation below is a slight modification of \cite[Section 7]{BowenIntro}.
	
	Let $\alpha \colon \G \pacts V$ be a pseudo-action of $\G$. For a function $f \colon V \to k$, define the map $\pi_f \colon V \to k^\G$ via
	\[
	\pi_f(v)(\gamma) \defeq f(\gamma \cdot v) \qquad \text{for all } v\in V \text{ and } \gamma \in \G.
	\]
	Note that if $\alpha$ is an \emph{action}, then the map $\pi_f \colon V \to k^\G$ is equivariant; in general, we have
	\[
		(\gamma \cdot \pi_f(v))(\delta) = \pi_f(\gamma \cdot v)(\delta) \qquad \text{whenever $\gamma$, $\delta \in \G$ and $v$ is $\set{\gamma, \delta, \delta\gamma}$-proper.}
	\]
	Let $X \subseteq k^\G$ be a subshift and suppose that the set $V$ is finite. An \emphd{$(\epsilon, F)$\=/approximate $X$\=/coloring} of $\alpha$, where $\epsilon > 0$ and $F \in \fins{\G}$, is a function $f \colon V \to k$ such that
	\[
	|\set{v \in V \,:\, \rest{\pi_f(v)}{F} \in X_F}| \geq  (1-\epsilon)|V|.
	\]
	The set of all $(\epsilon, F)$-approximate $X$-colorings of $\alpha$ is denoted by $\Col \epsilon F \alpha X$. Let
	\[
	h_{\epsilon, F}(X, \alpha) \,\defeq \, \frac{\log_2|\Col \epsilon F \alpha X|}{|V|}.
	\]
	If $\Col \epsilon F \alpha X = \0$, then, by definition, $h_{\epsilon, F}(X, \alpha) \defeq -\infty$.
	
	Now assume that $\G$ is sofic and let $\Sigma = (\alpha_n)_{n=0}^\infty$ be a sofic approximation to $\G$. The \emphd{\ep{topological} entropy} $h(X, \Sigma)$ of a nonempty subshift $X \subseteq k^\G$ with respect to $\Sigma$ is given by
	\[
	h(X, \Sigma) \,\defeq\, \inf_{\epsilon, F} \,\limsup_{n \to \infty} \, h_{\epsilon, F}(X, \alpha_n),
	\]
	where $\epsilon$ ranges over the positive reals and $F$ ranges over the finite subsets of $\G$. If $\G$ is amenable, then we have $h(X, \Sigma) = h(X)$ for any sofic approximation $\Sigma$~\cite{Bowen_amenable, KerrLi_amenable}. In general, however, the value of $h(X, \Sigma)$ may depend on $\Sigma$. Nevertheless, we have the following:
	
	\begin{prop}[{\cite[Propositions~10.28 and 10.29]{KerrLi_book}}]\label{prop:bounds_sofic}
		If $\G$ is sofic and $\Sigma$ is a sofic approximation to $\G$, then $h(k^\G, \Sigma) = \log_2 k$ and $h(X, \Sigma) < \log_2 k$ for any subshift $X \varsubsetneq k^\G$.
	\end{prop}
	
	Note that Proposition~\ref{prop:bounds_amenable} is a special case of Proposition~\ref{prop:bounds_sofic}.

	\section{Main definitions and results}\label{sec:main_result}
	
	
	\subsection{{Width and pointwise width}}\label{subsec:width}
	
	For $\phi \in \finf{\G}{k}\setminus \set{\0}$, let
	\[
		\d(U_\phi) \defeq 2^{-|\phi|}.
	\]
	Note that we have $0 < \d(U) \leq 1/2$ for every basic open set $U$. The value $\d(U)$ is preserved by the shift action and thus can be viewed as a shift-invariant alternative to the diameter $\mathrm{diam}(U)$. For a family $\U$ of basic open sets and a parameter $h \in [0;+\infty)$, let $\rho_h(\U) \defeq \sum_{U \in \U} \d(U)^h$ and define
	\[
		\w(\U) \defeq\inf\set{h \in [0;+\infty) \,:\, \rho_h(\U) \leq 1}.
	\]
	If the family $\U$ is finite and nonempty, then $\rho_h(\U)$, viewed as a function of $h$, is continuous and strictly decreasing. Thus, $\w(\U)$ for such $\U$ is equal to the unique $h \in [0;+\infty)$ with $\rho_h(\U) = 1$.
	
	A \emphd{cover} of a set $X \subseteq k^\G$ is a family $\U$ of basic open sets such that $X \subseteq \bigcup \U$. The \emphd{width} of $X$, denoted $\w(X)$, is defined via
	\[
		\w(X) \defeq  \inf \set{\w(\U)\,:\, \U \text{ is a cover of } X}.
	\]
	Notice the close analogy between this definition and that of $\dim_H(X)$ (see also Proposition~\ref{prop:comparison}\ref{item:Hausdorff}). We will frequently use the fact that, since the space $k^\G$ is compact, to determine $\w(X)$ for a {closed} subset $X \subseteq k^\G$ it is enough to only consider {finite} covers of $X$.
	
	The \emphd{pointwise width} of a set $X \subseteq k^\G$, denoted $\w^\ast(X)$, is given by
	\[
	\w^\ast(X) \defeq  \inf_{x \in X} \min \set{\w(\overline{\G \cdot x}), \, \w(\overline{x \cdot \G})}.
	\]
	Technically, we have $\w^\ast(\0) = +\infty$ (even though $\w(\0) = 0$).
	
	
	\begin{prop}\label{prop:easy}The following statements are valid:
		\mbox{}\begin{enumerate}[label=\normalfont\ep{\roman*},align=left,labelindent=\parindent,leftmargin=*,itemsep=\parskip]
			\item\label{item:subsets} If $Y \subseteq X \subseteq k^\G$, then $\w(Y) \leq \w(X)$ and $\w^\ast(Y) \geq \w^\ast(X)$.

			\item\label{item:bounds_for_subshifts} If $X \subseteq k^\G$ is a nonempty subshift, then $\w^\ast(X) \leq  \w(X)$.
		\end{enumerate}	
	\end{prop}
	\begin{proof}[{\textsc{Proof}}]
		Part \ref{item:subsets} is clear, and for part \ref{item:bounds_for_subshifts}, notice that for every point $x \in X$, we have $\overline{\G \cdot x} \subseteq X$, hence $\w^\ast(X) \leq \w(\overline{\G \cdot x}) \leq \w(X)$.
	\end{proof}
	
	\begin{prop}
		We have $\w(k^\G) = \log_2 k$ and $\w(X) < \log_2 k$ for any closed set $X \varsubsetneq k^\G$.
	\end{prop}
	\begin{proof}[{\textsc{Proof}}]
		Let $\nu$ denote the uniform probability measure on~$k$ and let $X \subseteq k^\G$ be a closed set. Since the product measure $\nu^\G$ on $k^\G$ is regular, we have
		\begin{align*}
			\nu^\G(X) &= \inf \set{\nu^\G(U) \,:\, U \subseteq k^\G \text{ is an open set with } U \supseteq X} \\
			&= \inf \set{\nu^\G (\bigcup \U) \,:\, \U \text{ is a finite cover of } X}.
		\end{align*}
		Since every finite family of basic open subsets of $k^\G$ admits a finite refinement consisting of pairwise disjoint basic open sets, we conclude that
		\begin{align}
			\nu^\G(X) &= \inf \set{\textstyle\sum\nolimits_{U \in \U} \nu^\G(U) \,:\, \U \text{ is a finite cover of } X} \nonumber \\
			&= \inf \set{\textstyle\sum\nolimits_{U \in \U} \d(U)^{\log_2 k} \,:\, \U \text{ is a finite cover of } X} \nonumber \\
			&=\inf\set{\rho_{\log_2 k}(\U) \,:\, \U \text{ is a finite cover of } X}.\label{eq:measure}
		\end{align}
		The desired conclusion now follows since $\nu^\G(k^\G) = 1$ and $\nu^\G(X) < 1$ if $X \neq k^\G$.
	\end{proof}
	
	The next proposition confirms the importance of width and pointwise width as notions of size:
	
	\begin{prop}\label{prop:comparison}
		{If $X \subseteq k^\G$ is a subshift, then:}
		\begin{enumerate}[label=\normalfont\ep{\roman*},align=left,labelindent=\parindent,leftmargin=*,itemsep=\parskip]
			\item\label{item:Hausdorff} {the Hausdorff dimension of $X$ is at least $\w(X)$;}
			\item\label{item:X_F/F} for every set $F \in \fins{\G} \setminus \set{\0}$, we have $\log_2 |X_F|/|F| \geq \w(X)$;
			\item\label{item:amenable} {if $\G$ is amenable, then the entropy of $X$ is a least $\w(X)$;}
			\item\label{item:free} {if $\w^\ast(X) > (1/2)\log_2 k$, then $X$ is free;}
			\item\label{item:open} {if $U \subseteq k^\G$ is a shift-invariant open set and $\w^\ast(X) > \w(k^\G \setminus U)$, then $X \subseteq U$.}
		\end{enumerate}
	\end{prop}
	\begin{proof}[{\textsc{Proof}}]
		\ref{item:Hausdorff} Let $\Ball(x,r)$ denote the open ball of radius $r > 0$ centered at a point $x \in k^\G$. If $n \in \N$ is such that $2^{-n-1} < r \leq 2^{-n}$, then $\Ball(x, r)$ is a basic open set with $\d(\Ball(x, r)) = \mathrm{diam}(\Ball(x, r)) = 2^{-n-1}$, and the desired result follows.
		
		\ref{item:X_F/F} The family $\set{U_\phi \,:\, \phi \in X_F}$ is a cover of $X$ with $\w(\set{U_\phi \,:\, \phi \in X_F}) = \log_2|X_F|/|F|$.
		
		\ref{item:amenable} Follows from \ref{item:X_F/F}.
			
		\ref{item:free} It is enough to prove that $\w(\overline{x \cdot \G}) \leq (1/2) \log_2 k$ for every point $x \in k^\G$ with $\St_\G(x) \neq \set{\mathbf{1}}$. To that end, suppose that $\mathbf{1} \neq \gamma \in \St_\G(x)$. For each $i < k$, let $\phi_i \colon \set{\mathbf{1}, \gamma} \to k$ be the map given by $\phi_i (\mathbf{1}) = \phi_i(\gamma) \defeq i$. Then $\set{U_{\phi_i} \,:\,i < k}$ is a cover of $\overline{x \cdot \G}$ and $\w(\set{U_{\phi_i} \,:\,i < k}) = (1/2) \log_2 k$.
		
		\ref{item:open} For any $x \in k^\G \setminus U$, we have $\overline{\G \cdot x} \subseteq k^\G \setminus U$, and hence $\w(\overline{\G \cdot x}) \leq \w(k^\G \setminus U)$.
	\end{proof}
	
	
	\subsection{{Breadth}}\label{subsec:elasticity}
	
	For a basic open set $U$ and a parameter $h \in (0; + \infty)$, let
	\[
		\sigma_h(U) \defeq \log_2 \d(U) \cdot \log_2 (1 - \d(U)^h).
	\]
	Note that both $\log_2\d(U)$ and $\log_2 (1- \d(U)^h)$ are negative, so $\sigma_h(U) > 0$. It is often useful to keep in mind that, when $\d(U)$ is small, we have
	\begin{equation}\label{eq:asymptotics}
	\sigma_h(U) \,\approx\, \log_2e \cdot |\log_2\d(U)| \cdot \d(U)^h.
	\end{equation}
	If $U = U_\phi$ for $\phi \in \finf{\G}{k} \setminus \set{\0}$, then \eqref{eq:asymptotics} can be rewritten as
	\[
		\sigma_h(U) \,\approx\, \log_2e \cdot |\phi| \cdot 2^{-h|\phi|}.
	\]
	For large $|\phi|$, the ``main'' term in the above expression is $2^{-h|\phi|}$, which is equal to $\d(U)^h$. In other words, it is usually safe to think of $\sigma_h(U)$ as ``almost'' equal to $\d(U)^h$, modulo a small perturbation.
	
	For a family $\U$ of basic open sets and $h \in (0;+\infty)$, let $\sigma_h(\U) \defeq \sum_{U \in \U} \sigma_h(U)$ and define
	\begin{equation}\label{eq:b_U}
		\b(\U) \defeq\sup\set{h \in (0;+\infty) \,:\, h + \sigma_h(\U) < \log_2k}.
	\end{equation}
	The value $\sigma_h(\U)$ is non-increasing as a function of $h$ (we cannot say that it is \emph{strictly} decreasing, but only because it may be infinite). Due to this fact, the expression $h + \sigma_h(\U)$ appearing in \eqref{eq:b_U} is not, in general, a monotone function of $h$. By convention, if $h + \sigma_h(\U) \geq \log_2 k$ for all $h \in (0;+\infty)$, then we set $\b(\U) \defeq 0$.
	
	An \emphd{action-cover} of a set $W \subseteq k^\G$ is a family $\U$ of basic open sets such that $W \subseteq \bigcup (\G \cdot \U)$, i.e., the translates of the sets in $\U$ cover $W$. The \emphd{breadth} of a set $X \subseteq k^\G$, denoted $\b(X)$, is given by
	\[
	\b(X) \defeq \sup \set{\b(\U) \,:\, \U \text{ is an action-cover of } k^\G \setminus X}.
	\]
	In contrast to $\w(X)$, to determine $\b(X)$ for a subshift $X$ we typically have to allow infinite families~$\U$. As mentioned in the introduction, the notion of breadth is made useful by the fact that a lower bound on $\b(X)$ can be witnessed by a \emph{single} action-cover $\U$ of $k^\G \setminus X$. On the other hand, obtaining {upper} bounds on $\b(X)$ can be more difficult. Indeed, \emph{a priori} it is not even obvious that $\b(\0) = 0$. (However, this statement is true and is part of our main result.)

	\subsection{{The main result}}\label{subsec:proof_corl}
	
	At this point, after all the necessary definitions have been introduced, we restate our main result, for the reader's convenience:
	
	
	
	\begin{theo}\label{theo:main}
		Let $X \subseteq k^\G$ be a subshift such that $\b(X) > 0$. Then $\w(X) \geq \b(X)$; moreover, for any $h < \b(X)$, there exists a nonempty subshift $X' \subseteq X$ with the following properties:
		\begin{enumerate}[label=\normalfont\ep{\roman*},align=left,labelindent=\parindent,leftmargin=*, itemsep=\parskip]
			\item the pointwise width of $X'$ is at least $h$;
			\item\label{item:main:sofic} if $\G$ is sofic, then the entropy of $X'$ with respect to any sofic approximation is at least~$h$;
			\item there exist an invariant probability measure $\mu$ on $X'$ and a factor map \[\pi \colon ([0;1]^\G, \lambda^\G) \to (X', \mu).\]
		\end{enumerate}
	\end{theo}
	
	
		With Theorem~\ref{theo:main} in hand, it is easy to derive Theorem~\ref{theo:corl}:
		
	\begin{proof}[\textsc{Proof of Theorem~\ref{theo:corl}}]
		Fix a nonempty shift-invariant open set $U \subseteq k^\G$ and let $h < \log_2 k$. Without loss of generality, we may assume that
		\begin{equation}\label{eq:h_large}
			h > \max \set{(1/2)\log_2 k,\, \w(k^\G \setminus U)}.
		\end{equation}
		Since, trivially, $\b(k^\G) = \log_2 k$, Theorem~\ref{theo:main} applied to $k^\G$ yields a nonempty subshift $X \subseteq k^\G$ such that $\w^\ast(X) \geq h$ and, if $\G$ is sofic, the entropy of $X$ with respect to any sofic approximation is at least $h$. From Proposition~\ref{prop:comparison} and~\eqref{eq:h_large}, it follows that $X$ is free and $X \subseteq U$. Let $Y \subseteq X$ be an arbitrary minimal subshift. Since we also have $\w^\ast(Y) \geq h$, Proposition~\ref{prop:comparison} implies that the Hausdorff dimension and, if $\G$ is amenable, the entropy of $Y$ are at least $h$.
	\end{proof}

	\section{Proof of Theorem~\ref{theo:main}}\label{sec:proof}
	
	\subsection{{The main lemmas}}
	
	For the purposes of the proof, we split Theorem~\ref{theo:main} into two parts.
	
	\begin{lemma}\label{lemma:pointwise}
		Let $X \subseteq k^\G$ be a subshift such that $\b(X) > 0$. Then, for any $h < \b(X)$, there exists a subshift $X' \subseteq X$ such that $\b(X') \geq h$ and $\w^\ast(X') \geq h$.
	\end{lemma}
	
	Note that Lemma~\ref{lemma:pointwise} does not yet guarantee that $X' \neq \0$ (as $\w^\ast(\0) = + \infty$). This is taken care of in Lemma~\ref{lemma:using_LLL}:
	
	\begin{lemma}\label{lemma:using_LLL}
		Let $X \subseteq k^\G$ be a subshift such that $\b(X) > 0$. Then $X \neq \0$; moreover,
		\begin{enumerate}[label=\normalfont\ep{\roman*},align=left,labelindent=\parindent,leftmargin=*, itemsep=\parskip]
			\item\label{item:lemma:sofic} if $\G$ is sofic, then the entropy of $X$ with respect to any sofic approximation is at least~$\b(X)$;
			\item\label{item:lemma:factor} there exist an invariant probability measure $\mu$ on $X$ and a factor map \[\pi \colon ([0;1]^\G, \lambda^\G) \to (X, \mu).\]
		\end{enumerate}
	\end{lemma}
	
	It is clear that Lemmas~\ref{lemma:pointwise} and \ref{lemma:using_LLL} combined yield Theorem~\ref{theo:main}.
	We prove Lemma~\ref{lemma:pointwise} in \S\ref{subsec:pointwise} by constructing the required subshift $X'$ explicitly. The proof of Lemma~\ref{lemma:using_LLL} crucially relies on the~LLL. We briefly review the~LLL in \S\ref{subsec:LLL} and then prove Lemma~\ref{lemma:using_LLL} in \S\ref{subsec:using_LLL}.
	
	\subsection{{Proof of Lemma~\ref{lemma:pointwise}}}\label{subsec:pointwise}
	
	\begin{claim}\label{claim:additions}
		Let $\mathscr{F}$ be a finite family of basic open sets with $\w(\mathscr{F}) < h$. Then, for any $\epsilon > 0$, there exist families $\mathscr{V}$ and $\mathscr{W}$ of basic open sets such that
		\[
			\sigma_h(\mathscr{V}) \,< \, \epsilon \qquad \text{and} \qquad \sigma_h(\mathscr{W}) \,< \, \epsilon,
		\]
		and for all $x \in k^\G \setminus \bigcup \mathscr{V}$ and $y \in k^\G \setminus \bigcup\mathscr{W}$, we have
		\[
			\G \cdot x \,\not\subseteq\, \bigcup \mathscr{F} \qquad \text{and} \qquad y \cdot \G \,\not\subseteq\, \bigcup \mathscr{F}.
		\]
	\end{claim}
	\begin{proof}[\textsc{Proof}]
		Below we only describe the construction of the family $\mathscr{V}$, as the family $\mathscr{W}$ is built in virtually the same way, the only difference being the use of the right instead of the left shift action.
		
		Let $\Phi \subset \finf{\G}{k}$ be the \ep{finite} set such that $\mathscr{F} = \set{U_\phi \,:\, \phi \in \Phi}$ and let $\epsilon > 0$. Let $N \in \N$ be a large integer (to be chosen later). Since $\G$ is infinite, we can find $N$ elements $\gamma_1$, \ldots, $\gamma_N \in \G$ such that for all $\phi$, $\psi \in \Phi$ and $1\leq i < j \leq N$, we have $\dom(\phi)\gamma_i \cap \dom(\psi)\gamma_j = \0$.
		This implies that for all $U_1$, \ldots, $U_N \in \mathscr{F}$, the set $\bigcap_{i=1}^N(\gamma_i^{-1} \cdot U_i)$ is basic open with
		\begin{equation}\label{eq:product}
		\d\left(\bigcap_{i=1}^N(\gamma_i^{-1} \cdot U_i)\right) \,=\, \prod_{i = 1}^N \d(U_i).
		\end{equation}
		We claim that the family
		$\mathscr{V} \defeq \set{\,\textstyle\bigcap\nolimits_{i=1}^N(\gamma_i^{-1} \cdot U_i) \,:\, U_1, \ldots, U_N \in \mathscr{F}}
		$ is as desired.
		
		Suppose that some $x \in k^\G \setminus \bigcup \mathscr{V}$ satisfies $\G \cdot x \subseteq \bigcup \mathscr{F}$. Then we can choose $U_i \in \mathscr{F}$ for each $1 \leq i \leq N$ so that $\gamma_i \cdot x \in U_i$. But this yields $x \in \bigcap_{i=1}^N(\gamma_i^{-1} \cdot U_i) \in \mathscr{V}$, which is a contradiction. Hence, it only remains to show that, if $N$ is large enough, then $\sigma_h(\mathscr{V}) < \epsilon$. To that end, consider an arbitrary sequence $U_1$, \ldots, $U_N \in \mathscr{F}$. By \eqref{eq:product}, we have
		\[
		\sigma_h\left(\bigcap_{i=1}^N(\gamma_i^{-1} \cdot U_i)\right) \,=\, \log_2 \prod_{i = 1}^N \d(U_i) \cdot \log_2\left(1-\prod_{i = 1}^N \d(U_i)^h\right).
		\]
		Let $c_1 \defeq \max\set{|\log_2\d(U)|\,:\,U \in \mathscr{F}}$ and $c_2 \defeq 2^h |\log_2(1-2^{-h})|$. (Note that the values $c_1$ and $c_2$ do not depend on $N$.) We have
		\[
		\left|\log_2 \prod_{i = 1}^N \d(U_i) \right| \,=\, \left|\sum_{i=1}^N\log_2 \d(U_i) \right| \,\leq\, c_1N,
		\]
		and, since $|\log_2(1-a)| \leq c_2 \cdot a$ for all $a \in [0;2^{-h}]$, we also have
		\[
		\log_2\left(1-\prod_{i = 1}^N \d(U_i)^h\right) \,\leq\, c_2 \cdot \prod_{i=1}^N \d(U_i)^h.
		\]
		Therefore,
		\[
		\sigma_h\left(\bigcap_{i=1}^N(\gamma_i^{-1} \cdot U_i)\right) \,\leq\, c_1c_2 \cdot N \cdot \prod_{i=1}^N \d(U_i)^h.
		\]
		Since $\w(\mathscr{F}) < h$, we have $\rho_h(\mathscr{F}) < 1$, and hence
		\[
		\sigma_h(\mathscr{V}) \,\leq\, c_1c_2 \cdot N \cdot \sum_{U_1, \ldots, U_N \in \mathscr{F}} \,\prod_{i=1}^N \d(U_i)^h \,=\, c_1c_2 \cdot N \cdot \rho_h(\mathscr{F})^N  \,\xrightarrow[N \to \infty]{}\, 0. \qedhere
		\]
	\end{proof}
	
	Let $X \subseteq k^\G$ be a subshift such that $\b(X) > 0$ and let $h < \b(X)$. We may assume that $h > 0$ and that there exists an action-cover $\U$ of $k^\G \setminus X$ such that $h + \sigma_h(\U) < \log_2 k$. Let $\mathscr{F}_0$, $\mathscr{F}_1$, \ldots{} be an arbitrary enumeration of all the finite families $\mathscr{F}$ of basic open sets satisfying $\w(\mathscr{F}) < h$. For each $n \in \N$, let $\mathscr{V}_n$ and $\mathscr{W}_n$ be the families given by Claim~\ref{claim:additions} applied to the family  $\mathscr{F}_n$ with
	\[
		\epsilon_n \defeq \frac{\log_2 k - h - \sigma_h(\U)}{2^{n+2}}.
	\]
	Set $\U' \defeq \U \cup \bigcup_{n = 0}^\infty (\mathscr{V}_n \cup \mathscr{W}_n)$. Then the subshift $X' \defeq k^\G \setminus (\G \cdot \U')$ is as desired. Indeed, since $\U$ is an action-cover of $k^\G \setminus X$, we have $X' \subseteq k^\G \setminus (\G \cdot \U) \subseteq X$. Since
	\begin{align*}
		h \,+\, \sigma_h(\U) \,&+\, \sum_{n=0}^\infty (\sigma_h(\mathscr{V}_n) \,+\, \sigma_h(\mathscr{W}_n)) \\
		<\, h \,+\, \sigma_h(\U) \,&+\, \sum_{n=0}^\infty \frac{\log_2 k - h - \sigma_h(\U)}{2^{n+1}} \,=\, \log_2 k,
	\end{align*}
	we conclude that $\b(X') \geq h$. Finally, if $x \in k^\G$ satisfies $\w(\overline{\G \cdot x}) < h$ or $\w(\overline{x \cdot \G}) < h$, then there exists an index $n \in \N$ such that $\G \cdot x \subseteq \bigcup \mathscr{F}_n$ or $x \cdot \G \subseteq \bigcup \mathscr{F}_n$. By the choice of $\mathscr{V}_n$ and $\mathscr{W}_n$, such $x$ cannot belong to $X'$, and hence $\w^\ast(X') \geq h$. The proof of Lemma~\ref{lemma:pointwise} is complete.

	\subsection{{The Lov\'asz Local Lemma}}\label{subsec:LLL}
	
	The LLL was first introduced by Erd\H os and Lov\'asz in~\cite{EL}. It is usually stated probabilistically:
	
	\begin{theo}[{Erd\H os--Lov\'asz, \textls{Lov\'asz Local Lemma} \cite[Lemma 5.1.1]{AS}}]\label{theo:LLL}
		Let $\B$ be a finite collection of random events in a probability space $\Omega$. For each $B \in \B$, let $\Nbhd(B) \subseteq \B \setminus \set{B}$ be a subset such that $B$ is independent from the algebra generated by $\B \setminus (\Nbhd(B) \cup \set{B})$. Suppose that a function $\omega \colon \B \to [0;1)$ satisfies
		\[
			\mathbb{P}[B] \,\leq\, \omega(B) \prod_{B' \in \Nbhd(B)} (1- \omega(B')) \qquad \text{for all } B \in \B.
		\]
		Then
		\[
			\mathbb{P}\left[\bigwedge_{B \in \B} \neg B\right] \,\geq\, \prod_{B \in \B} (1- \omega(B)) \,>\, 0.
		\]
	\end{theo}
	
	The LLL is often used in the form of the following corollary. Let $X$ be a set and let $\Phi \subseteq \finf{X}{k}$. Let $\Forb(\Phi)$ denote the set of all maps $f \colon X \to k$ such that $f \not\supseteq \phi$ for all $\phi \in \Phi$. For each $\phi \in \Phi$, let
	\[
		\Nbhd(\phi, \Phi) \,\defeq\, \set{\psi \in \Phi \,:\, \dom(\phi) \cap \dom(\psi) \neq \0}. 
	\]
	We say that $\Phi$ is \emphd{correct \ep{for the LLL}} if there is a function $\omega \colon \Phi \to [0;1)$ such that
	\[
		k^{-|\phi|} \,\leq\, \omega(\phi) \prod_{\psi \in \Nbhd(\phi, \Phi)} (1 - \omega(\psi)) \qquad \text{for all } \phi \in \Phi.
	\]
	In this case $\omega$ is called a \emphd{witness} to the correctness of $\Phi$.
	
	\begin{corl}\label{corl:LLL}
		Let $X$ be a set and let $\Phi \subseteq \finf{X}{k}$. If the set $\Phi$ is correct, then $\Forb(\Phi) \neq \0$. Furthermore, if $X$ is finite and $\omega \colon \Phi \to [0;1)$ is a witness to the correctness of $\Phi$, then
		\[
			|\Forb(\Phi)| \,\geq\, k^{|X|} \prod_{\phi \in \Phi} (1 - \omega(\phi)).
		\]
	\end{corl}
	
	For finite $X$, Corollary~\ref{corl:LLL} follows from Theorem~\ref{theo:LLL} by taking $\Omega$ to be the set $k^X$ equipped with the uniform probability measure (see, e.g., \cite[41]{MR} for more details). The infinite case is derived from the finite one via a straightforward compactness argument.
	
	Recently, Moser and Tardos~\cite{MT} developed an algorithmic approach to the~LLL that has led to a large amount of work concerning various effective versions of the~LLL. A~salient example is the \emph{computable} version of the~LLL due to Rumyantsev and Shen~\cite{RSh}. In another direction, several \emph{measurable} versions of the~LLL have been established \cite{Bernshteyn, CGMPT}. Here we will use a measurable version of the~LLL for group actions from \cite{Bernshteyn}.
	
	Suppose that $\Phi \subseteq \finf{\G}{k}$. Then we have $\Forb(\Phi) \,=\, k^\G \setminus \bigcup_{\phi \in \Phi} U_\phi$. In particular, if the set $\Phi$ is shift-invariant, then $\Forb(\Phi)$ is a subshift. Let $\alpha \colon \G \acts (X, \mu)$ be a measure-preserving action of $\G$ on a probability space $(X, \mu)$. Given a shift-invariant set $\Phi \subseteq \finf{\G}{k}$, a \emphd{measurable solution} to $\Phi$ over $\alpha$ is a measurable function $f \colon X \to k$ such that for $\mu$-almost all $x \in X$, the map
	\[
		\pi_f(x) \colon \G \to k \colon \gamma \mapsto f(\gamma \cdot x)
	\]
	belongs to $\Forb(\Phi)$.
	
	\begin{theo}[{\cite[Corollary~6.7]{Bernshteyn}}]\label{theo:meas_LLL}
		Let $\Phi \subseteq \finf{\G}{k}$ be a correct shift-invariant set. Then the shift action $\G \acts ([0;1]^\G, \lambda^\G)$ admits a measurable solution to $\Phi$.
	\end{theo}
	
	\begin{corl}\label{corl:measure}
		Let $\Phi \subseteq \finf{\G}{k}$ be a correct shift-invariant set. Then there exist an invariant probability measure $\mu$ on $\Forb(\Phi)$ and a factor map \[\pi \colon ([0;1]^\G, \lambda^\G) \to (\Forb(\Phi), \mu).\]
	\end{corl}
	\begin{proof}[\textsc{Proof}]
		Let $f \colon [0;1]^\G \to k$ be a measurable solution to $\Phi$ given by Theorem~\ref{theo:meas_LLL}. We may then take $\pi \defeq \pi_f$ and $\mu \defeq (\pi_f)_\ast(\lambda^\G)$.
	\end{proof}
	
	Theorem~\ref{theo:meas_LLL} is a special case of \cite[Theorem~6.6]{Bernshteyn}, whose full statement is somewhat technical and will not be needed here. Roughly speaking, \cite[Theorem~6.6]{Bernshteyn} asserts that any combinatorial argument that proceeds via a series of iterative applications of the~LLL can be preformed in a measurable fashion over the shift action $\G \acts ([0;1]^\G; \lambda^\G)$.
	
	\subsection{{Proof of Lemma~\ref{lemma:using_LLL}}}\label{subsec:using_LLL}
	
	\begin{claim}\label{claim:breadth_LLL}
		Let $\U$ be a family of basic open sets and let $h \in (0;+\infty)$ be such that
		\begin{equation}\label{eq:bound0}
			h + \sigma_h(\U) < \log_2 k.
		\end{equation}
		Let $\Phi \subseteq \finf{\G}{k}$ be the set such that $\U = \set{U_\phi \,:\, \phi \in \Phi}$. For each $\phi \in \G \cdot \Phi$, define
		\[
		\omega(\phi) \,\defeq\, 2^{-h|\phi|}.
		\]
		Then $\omega$ is a witness to the correctness of $\G \cdot \Phi$.
	\end{claim}
	\begin{proof}[\textsc{Proof}]
		Since $\omega$ is invariant under the shift action $\G \acts \G \cdot \Phi$, we only have to verify that
		\[
		k^{-|\phi|} \,\leq\, \omega(\phi) \prod_{\psi \in \Nbhd(\phi, \G \cdot \Phi)} (1 - \omega(\psi)) \qquad \text{for all } \phi \in \Phi.
		\]
		Let $\phi \in \Phi$. By definition, $\Nbhd(\phi, \G \cdot \Phi)$ is the set of all products of the form $\delta \cdot \psi$, where $\psi \in \Phi$ and $\delta \in \G$, with the property that $\dom(\phi) \cap \dom(\delta \cdot \psi) \neq \0$. This is equivalent to $\delta \in \dom(\phi)^{-1} \dom(\psi)$, so, for each choice of $\psi \in \Phi$, there are at most $|\dom(\phi)^{-1} \dom(\psi)| \leq |\phi||\psi|$ possible choices for $\delta \in \G$. Using this observation together with the shift-invariance of $\omega$, we obtain
		\[
		\omega(\phi) \prod_{\psi \in \Nbhd(\phi, \G \cdot \Phi)} (1 - \omega(\psi)) \,\geq\, \omega(\phi) \prod_{\psi \in \Phi} (1 - \omega(\psi))^{|\phi||\psi|}.
		\]
		It remains to show that
		\begin{equation}\label{eq:bound1}
		k^{-|\phi|}\,\leq\, \omega(\phi) \prod_{\psi \in \Phi} (1 - \omega(\psi))^{|\phi||\psi|}.
		\end{equation}
		Plugging the definition of $\omega$ into \eqref{eq:bound1}, we get
		\[
			k^{-|\phi|}\,\leq\, 2^{-h|\phi|} \prod_{\psi \in \Phi} (1 - 2^{-h|\psi|})^{|\phi||\psi|},
		\]
		which is equivalent to
		\[
		k^{-1} \,\leq\, 2^{-h} \prod_{\psi \in \Phi}(1-2^{-h|\psi|})^{|\psi|}.
		\]
		Taking the logarithm on both sides turns the last inequality into
		\[
		\log_2k \,\geq\, h - \sum_{\psi \in \Phi} |\psi| \cdot \log_2(1 - 2^{-h|\psi|}).
		\]
		But $|\psi| = - \log_2(\d(U_\psi))$ and $2^{-h|\psi|} = \d(U_\psi)^h$, so
		\[
		- \sum_{\psi \in \Phi} |\psi| \cdot \log_2(1 - 2^{-h|\psi|}) \,=\, \sum_{U \in \U} \log_2(\d(U)) \cdot \log_2(1-\d(U)^h) \,=\, \sigma_h(\U),
		\]
		and we are done by \eqref{eq:bound0}.
	\end{proof}
	
	Let $X \subseteq k^\G$ be a subshift with $\b(X) > 0$ and consider any action-cover $\U$ of $k^\G \setminus X$ with $\b(\U) > 0$. Let $h \in (0;+\infty)$ be such that $h + \sigma_h(\U) < \log_2k$. Note that $\U$ and $h$ can be chosen so that $h$ is as close to $\b(X)$ as desired. Let $\Phi \subseteq \finf{\G}{k}$ be the set such that $\U = \set{U_\phi \,:\, \phi \in \Phi}$. According to Claim~\ref{claim:breadth_LLL}, the set $\G \cdot \Phi$ is correct for the~LLL. Corollary~\ref{corl:LLL} then implies that $\Forb(\G \cdot \Phi) \neq \0$; furthermore, according to Corollary~\ref{corl:measure}, there exist an invariant probability measure $\mu$ on $\Forb(\G \cdot \Phi)$ and a factor map $\pi \colon ([0;1]^\G, \lambda^\G) \to (\Forb(\G \cdot \Phi), \mu)$. Since $\Forb(\G \cdot \Phi) = k^\G \setminus \bigcup (\G \cdot \U) \subseteq X$, we conclude that $X \neq \0$ and part \ref{item:lemma:factor} of Lemma~\ref{lemma:using_LLL} holds.
	
	It remains to verify that if $\G$ is sofic, then the entropy of $X$ with respect to any sofic approximation is at least $\b(X)$. In fact, we will show that the entropy of $\Forb(\G \cdot \Phi)$ is at least $h$, which will yield the desired result as $\Forb(\G \cdot \Phi) \subseteq X$ and $h$ can be made arbitrarily close to $\b(X)$. The idea is simple: Given a pseudo-action $\alpha \colon \G \pacts V$ on a finite set $V$, we ``copy'' $\G \cdot \Phi$ over to $V$ and build a set $\Phi_\alpha \subseteq \finf{V}{k}$ such that every map in $\Forb(\Phi_\alpha)$ is an approximate $\Forb(\G \cdot \Phi)$-coloring of $\alpha$; then we apply the LLL to obtain a lower bound on $|\Forb(\Phi_\alpha)|$. In the remainder of the proof, we work out the technical details of this approach.

	Let $\epsilon > 0$ and let $F \in \fins{\G}\setminus \set{\0}$. Recall that for a subshift $Y \subseteq k^\G$, the set $Y_F$ is defined by
	\[
	Y_F \defeq \set{\phi \in k^F \,:\, Y \cap U_\phi \neq \0}= \set{\rest{y}{F} \,:\, y \in Y}.
	\]
	By compactness, we can find a finite set $S \in \fins{\G}$ such that
	\begin{equation}\label{eq:compactness}
		\Forb(\G \cdot \Phi)_F \,=\, \Forb(S \cdot (\Phi \cap \finf{S}{k}))_F.
	\end{equation}
	We may assume that the set $S$ is symmetric and contains $\mathbf{1}$. For each $n \in \N$, let
	\[
		S^n \defeq \set{\gamma_1 \cdots \gamma_n \,:\, \gamma_1, \ldots, \gamma_n \in S}.
	\]
	Let $\alpha \colon \G \pacts V$ be an $(\epsilon, S^4)$-faithful pseudo-action of $\G$ on a finite set $V$. We will show that \[h_{\epsilon, F}(\Forb(\G \cdot \Phi), \alpha) \geq h.\] For each $\phi \in \finf{\G}{k}$ and $v \in \Prop_{\dom(\phi)} (\alpha)$, define the map $\phi_v \in \finf{V}{k}$ by
	\[
		\dom(\phi_v) \defeq \dom(\phi) \cdot v \qquad \text{and} \qquad \phi_v(\gamma \cdot v) \defeq \phi(\gamma) \text{ for all }\gamma \in \dom(\phi),
	\]
	and let
	\[
		\Phi_\alpha \defeq \set{\phi_v \,:\, \phi \in \Phi \cap \finf{S}{k},\, v \in \Prop_{S^3}(\alpha)}.
	\]
	
	\begin{claim}\label{claim:col}
		We have $\Forb(\Phi_\alpha) \subseteq \Col \epsilon F \alpha {\Forb(\G \cdot \Phi)}$.
	\end{claim}
	\begin{proof}[\textsc{Proof}]
	Let $f \in \Forb(\Phi_\alpha)$. Note that for any $v \in \Prop_{S^4}(\alpha)$, we have $S \cdot v \subseteq \Prop_{S^3}(\alpha)$, and therefore $\pi_f(v) \in \Forb(S \cdot (\Phi \cap \finf{S}{k}))$. From \eqref{eq:compactness} we conclude
	\[
		|\set{v \in V \,:\, \rest{\pi_f(v)}{F} \in \Forb(\G \cdot \Phi)_F}| \,\geq\,|\Prop_{S^4}(\alpha)| \,\geq\, (1-\epsilon)|V|. \qedhere
	\]
	\end{proof}
	
	Recall that, according to Claim~\ref{claim:breadth_LLL}, the map
	\[
		\omega \colon \G \cdot \Phi \to [0;1) \colon \phi \mapsto 2^{-h|\phi|}
	\]
	is a witness to the correctness of $\G \cdot \Phi$. Define
	\[
		\omega_\alpha \colon \Phi_\alpha \to [0;1) \colon \psi \mapsto 2^{-h|\psi|}.
	\]
	
	\begin{claim}\label{claim:tau}
	 The map $\omega_\alpha$ is a witness to the correctness of $\Phi_\alpha$.
	\end{claim}
	\begin{proof}[\textsc{Proof}]
		Consider any $v \in \Prop_{S^3}(\alpha)$ and $\phi \in \Phi \cap \finf{S}{k}$. We will define an injective map \[\iota \colon \Nbhd(\phi_v, \Phi_\alpha) \to \Nbhd(\phi, \G \cdot \Phi)\] such that for all $\psi \in \Nbhd(\phi_v, \Phi_\alpha)$, we have $|\psi| = |\iota(\psi)|$. Since then we also have $\omega_\alpha(\psi) = \omega(\iota(\psi))$, the desired conclusion follows by Claim~\ref{claim:breadth_LLL}.
	
		Suppose that $\psi_u \in \Nbhd(\phi_v, \Phi_\alpha)$ for some $u \in \Prop_{S^3}(\alpha)$ and $\psi \in \Phi \cap \finf{S}{k}$. Choose arbitrary $\gamma \in \dom(\phi)$ and $\delta \in \dom(\psi)$ such that $\gamma \cdot v = \delta \cdot u$ and define
	\[
		\iota(\psi_u) \defeq (\gamma^{-1}\delta) \cdot \psi.
	\]
	Clearly, $\iota(\psi_u) \in \Nbhd(\phi, \G \cdot \Phi)$ since $\gamma \in \dom(\iota(\psi_u))$. Also, we have $|\psi_u| = |\psi| = |\iota(\psi_u)|$. Finally, the map $\iota$ is injective, since it is invertible:  $\psi_u = (\iota(\psi_u))_v$. Indeed, as $v$ and $u$ are both $S^3$-proper, for every $\zeta \in \dom(\psi)$, we have
	\[
		(\zeta \delta^{-1} \gamma) \cdot v \,=\, \zeta \cdot (\delta^{-1} \cdot (\gamma \cdot v)) \,=\, \zeta \cdot (\delta^{-1} \cdot (\delta \cdot u)) \,=\, \zeta \cdot u,
	\]
	so $\psi_u = (\iota(\psi_u))_v$, as claimed.
	\end{proof}
	
	From Claim~\ref{claim:tau} and Corollary~\ref{corl:LLL}, we obtain
	\begin{align*}
		|\Forb(\Phi_\alpha)| \,\geq \, k^{|V|} \prod_{\psi \in \Phi_\alpha} (1 - \omega_\alpha(\psi)) \,&\geq\, k^{|V|} \prod_{\phi \in \Phi \cap \finf{S}{k}} \, \prod_{v \in \Prop_{S^{3}}(\alpha)} (1 - \omega_\alpha(\phi_v)) \\
		&\geq\, k^{|V|} \prod_{\phi \in \Phi} (1 - 2^{-h|\phi|})^{|V|}.
	\end{align*}
	Therefore, by Claim~\ref{claim:col},
	\begin{align*}
		h_{\epsilon, F}(\Forb(\G \cdot \Phi), \alpha) \,&=\, \frac{\log_2|\Col \epsilon F \alpha {\Forb(\G \cdot \Phi)}|}{|V|} \\
		&\geq \, \frac{\log_2|\Forb(\Phi_\alpha)|}{|V|} \,\geq \, \log_2 k + \sum_{\phi \in \Phi} \log_2(1 - 2^{-h|\phi|}).
	\end{align*}
	But $2^{-h|\phi|} = \d(U_\phi)^h$ and $-\log_2(\d(U_\phi)) = |\phi| \geq 1$, so
	\begin{align*}
	\log_2 k + \sum_{\phi \in \Phi} \log_2(1 - 2^{-h|\phi|}) \,&=\, \log_2 k + \sum_{U \in \U} \log_2(1-\d(U)^h) \\
	&\geq\, \log_2k - \sum_{U \in \U} \log_2(\d(U)) \cdot \log_2(1-\d(U)^h) \\
	&=\, \log_2 k - \sigma_h(\U) \,>\, h,
	\end{align*}
	as desired.

	{\renewcommand{\markboth}[2]{}
		\printbibliography}

\end{document}